\documentclass{amsart}
\usepackage{amssymb}
\usepackage{amscd}
\usepackage[all,cmtip]{xy}
\usepackage{graphicx}
\usepackage{hyperref}

\providecommand{\tr}{\mathop{\rm tr}\nolimits}
\providecommand{\Pic}{\mathop{\rm Pic}\nolimits}
\providecommand{\chart}{\mathop{\rm char}\nolimits}
\providecommand{\fv}{\mathop{\rm f}\nolimits}
\providecommand{\di}{\mathop{\rm div}\nolimits}
\providecommand{\disc}{\mathop{\rm disc}\nolimits}
\providecommand{\lcm}{\mathop{\rm lcm}\nolimits}
\providecommand{\unr}{\mathop{\rm unr}\nolimits}
\providecommand{\Norm}{\mathop{\rm Norm}\nolimits}
\providecommand{\Hom}{\mathop{\rm Hom}\nolimits}
\providecommand{\Gal}{\mathop{\rm Gal}\nolimits}
\providecommand{\Spec}{\mathop{\rm Spec}\nolimits}
\providecommand{\sh}{\mathop{\rm sh}\nolimits}
\providecommand{\Aut}{\mathop{\rm Aut}\nolimits}
\providecommand{\Tv}{\mathop{\rm T}\nolimits}

\providecommand{\di}{\mathbb{Z}}
\providecommand{\Z}{\mathbb{Z}}
\providecommand{\C}{\mathbb{C}}
\providecommand{\F}{\mathbb{F}}
\providecommand{\R}{\mathbb{R}}
\providecommand{\new}{\mathrm{new}}

\theoremstyle{plain}
\newtheorem{theorem}{Theorem}[section]

\newtheorem{corollary}[theorem]{Corollary}

\newtheorem{lemma}[theorem]{Lemma}

\newtheorem{proposition}[theorem]{Proposition}

\theoremstyle{definition}
\newtheorem{definition}[theorem]{Definition}
\newtheorem{remark}[theorem]{Remark}

\newtheorem{example}[theorem]{Example}

\newtheorem{conv}{Convention}

\begin{document}

\title[Generating Picard groups]{Deterministically generating Picard groups of hyperelliptic curves over finite fields}
\author{Michiel Kosters}
\address{Mathematisch Instituut
P.O. Box 9512
2300 RA Leiden
the Netherlands}
\email{mkosters@math.leidenuniv.nl}
\urladdr{www.math.leidenuniv.nl/~mkosters}
\date{\today}
\thanks{I would to thank my PhD supervisor Hendrik Lenstra for his help.}


\begin{abstract}
Let $\epsilon>0$. In this article we will present a deterministic algorithm which does the following. The input is a hyperelliptic curve $C$ of genus $g$ over a finite field $k$ of cardinality $q$ given by $y^2+h(x)y=f(x)$ such that the $x$-coordinate map is ramified at $\infty$. In time $O(g^{2+\epsilon} q^{1/2+\epsilon})$ the algorithm outputs a set of generators of the Picard group $\Pic^0_k(C)$. This extends results which others have obtained when $g=1$.

In this article we introduce a combinatorial tool, the `shape parameter', which we use together with character sum estimates from class field theory to deduce the statement.

\end{abstract}

\maketitle
\tableofcontents

\section{Introduction}

This article covers some of the results of my PhD thesis written under the supervision of Hendrik Lenstra at the Universiteit Leiden. For more details, we refer to the PhD thesis (\cite{KO6}).

An algorithmic problem in arithmetic geometry is to explicitly find the group structure of the Picard group of a curve of genus $g$ over a finite
field of
size $q=p^n$. A related problem is to find a generating set of this Picard group. Let $\epsilon>0$. In this article we describe a deterministic way of finding a
generating set, when the curve is hyperelliptic, in time $O(g^{2+\epsilon} q^{1/2+\epsilon})$.

Let $C$ be a hyperelliptic curve of genus $g$ over a finite field $k$ of cardinality $q$ and characteristic $p$ given by an equation
$y^2+h(x)y=f(x)$. We require that $(f,h)$ satisfies certain conditions (see Subsection \ref{t1}) and we assume that the natural projection map to the projective line by taking the $x$-coordinate is ramified at $\infty$. Our main theorem is the following.

\begin{theorem} \label{4cext}
For any $\epsilon>0$ there is a deterministic algorithm which on input a hyperelliptic curve $C$ of genus $g$ over a finite field $k$ of cardinality
$q$ outputs a set of generators of the Picard group $\Pic^0_k(C)$ in time $O(g^{2+\epsilon} q^{1/2+\epsilon})$ . 
\end{theorem}

Such a generating set can then be used in other algorithms to deterministically determine the group structure of $\Pic_k^0(C)$. 

Let us discuss one of the main ingredients of the proof of Theorem \ref{4cext}.
Let $\infty'$ the point above $\infty$. Let $\varphi_C: C(k) \to
\Pic^0_k(C)$ be the map given by $P \mapsto [P]-[\infty']$. 
For a subset $S$ of $k$ put $C_S=\{P \in C(k): x(P) \in S\}$. 
An interval $I$ of $k$ is a subset of
the form $B+\alpha[s,\ldots,s+r]$ where $B$ is an additive subgroup of $k$, $\alpha \in k$ and $s,r \in \Z_{\geq 0}$ (or more precisely, see Definition \ref{3c540}).

Kohel and Shparlinksi (\cite[Corollary 2]{SHP}) have shown the following for $g=1$. For $S$ an interval of $k$ of cardinality greater than $15(1+\log(p))q^{1/2}$ they deduce that $\langle \varphi_C(C_S) \rangle=\Pic_k^0(C)$. We generalize and improve their result in the following ways. This possible generalization was already suggested in \cite{SHP}

\begin{theorem} \label{4c845}
Assume that $\#C(k)>(2g-2)\sqrt{q}$. Assume that $p \neq 2$ or $p=2$ and $\deg(h)<g$. Let $S \subseteq k$ be a coset of a subgroup or an interval. Put $s=2$ if $p=2$ and $s=3$ if $p \neq 2$. Put
$t=1$ if $S$ is a coset of a subgroup and $t=2$ if $S$ is an interval which is not a coset of a subgroup. Assume that 
\begin{eqnarray*}
\#S \geq 2t(2g-2+s) \sqrt{q}.
\end{eqnarray*}
Then we have $\langle \varphi_C(C_S) \rangle=\Pic_k^0(C)$. 
\end{theorem}

The above theorem improves the results of \cite{SHP} in the following ways.

\begin{itemize}
 \item We allow hyperelliptic curves of any genus.
 \item We obtain similar theorems for subsets of $S \subseteq k$ which are not intervals or subgroups. For this reason we introduce the notion of the shape parameter of such a subset $S$.
 \item Our constants, as can be seen above, are a bit better. This improvement is already partially
suggested in \cite{SHP}. Furthermore, we do not have a $(1+\log{p})$ factor. This improvement is also suggested in \cite{SHP}. 
 \item We look at the case $p=2$ in the above theorem, even when $\deg(h)=g$. This case requires more work and there are exceptional cases. In
\cite{SHP}, this case is avoided by finding a similar result for the $y$-coordinate. In the end our estimates are
better when $p=2$, but there are exceptional sets coming from certain morphisms. Here is an example. 
Assume that $E$ is an elliptic curve over a finite field $k$ of characteristic $2$ given by 
$y^2+a_1xy+a_3=x^3+a_2x^2+a_4x+a_6$ with $a_1 \neq 0$. Then the map 
\begin{eqnarray*}
\psi_E: E(k) &\to& \F_2 \\
 P &\mapsto& \tr_{k/\F_2}((x(P)+a_2)/a_1^2) \\
\infty &\mapsto& 0 
\end{eqnarray*}
is a surjective morphism of groups with kernel $2E(k)$ (Proposition \ref{4c923}). Hence if we take $S=\{s \in k: \tr_{k/\F_2}((s+a_2)/a_1^2)=0\}$, a
coset of a subgroup $k$
of cardinality $q/2$, then $\langle P \in E(k): x(P) \in S \rangle=2E(k)$. 
\end{itemize}

In \cite{SHP} the authors use the aforementioned corollary (\cite[Corollary 2]{SHP}) to give a deterministic algorithm to find the group structure of the set of rational points of an elliptic curve over a finite field of size $q$ in $O(q^{1/2+\epsilon})$. By lack of good pairings, we use Theorem \ref{4c845} just to find a generating set of the Picard group. We deduce Theorem \ref{4cext} from Theorem \ref{4c845} by using intervals which are large enough.

The strategy of the proof of Theorem \ref{4c845} is the following. First we translate our problem to the calculation of certain character sums on the
finite abelian group $k^+
\times \Pic_k^0(C)$. We then construct, using class field theory, a finite geometric abelian extension of function
fields $M$ of $k(C)$ with group $G=k^+ \times \Pic_k^0(C)$, which for a point $P \in C(k) \setminus \{\infty'\}$ satisfies
$(P,M/k(C))=(x(P),[P]-[\infty'])
\in G$. Then using theorems from class field theory, we can estimate the character sums after we have calculated conductors of certain subextensions
of $M/k(C)$. In certain exceptional cases, our proof does not work. The extension $M/k(C)$ we obtain either has Galois
group which is smaller than $k^+ \times \Pic_k^0(C)$ or $M/k(C)$ is not geometric. With a bit more work, one can still work out these cases.

\section{Preliminaries}

\subsection{Curves and function fields}
We assume that the reader is familiar with the theory of curves and function fields (see for example \cite{ROS}, \cite{ST}). In this subsection we introduce some notation and recall some facts.

Let $k$ be a field. A function field over $k$ is a finitely generated field extension of $k$ of transcendence degree $1$. There is an anti-equivalence of categories between the category of normal projective curves over $k$ with finite morphisms and the category of function fields over $k$ with finite morphisms (see \cite[Proposition 3.13]{LIU}). A curve $C$ is mapped to its function field $k(C)$ and a map $C \to D$ of curves induces an inclusion $k(D) \subseteq k(C)$. We will mostly study normal projective curves by looking at their function fields. The set of non-generic points of such a curve $C$ correspond to the set of places $\mathcal{P}_{k(C)/k}$ of $k(C)$, that is, the valuation rings of $k(C)$ which contain $k$ but are not equal to $k(C)$. Note that $C(k)$ corresponds to the valuations subset of valuations of $\mathcal{P}_{k(C)/k}$ of degree $1$.

Let $K$ be a function field $k$. The full constant field of $K$ is the integral closure of $k$ in $K$. We say that $K$ is geometrically irreducible if the full constant field is $k$. The genus of $K$ is denoted by $g(K)$. Let $\di_k(K)$ be the free abelian group on $\mathcal{P}_{K/k}$. An element $D \in \di_k(K)$ is called a divisor of $K$. If $D$ is a divisor on $K$, we denote by $\deg_k(D)$ its $k$-degree. If $P \in \mathcal{P}_{K/k}$ we denote by $D_P \in \Z$ the coefficient of $D$ corresponding to $P$. The divisors of degree $0$ are denoted by $\di_k^0(K)$. An element $f$ of $K^*$ gives rise to a divisor of degree $0$, denoted by $(f)$. The Picard group, $\Pic^0_k(K)$, is defined by the exactness of the sequence
\begin{eqnarray*}
0 \to K^* \to \di_k^0(K) \to \Pic^0_k(K) \to 0.
\end{eqnarray*}
If $k$ is finite, then $\Pic^0_k(K)$ is finite.
If $L/K$ is a finite field extension, then by $\disc(L/K)$ we denote its discriminant.
Let $P \in \mathcal{P}_{L/k}$. Then by $P|_K$ we denote the restriction of $P$ to $K$. We set $\fv(P/P|_K)=\frac{\deg_k(P)}{\deg_k(P|_K)}$. We say that $L/K$ is geometric if the full constant fields of $L$ and $K$ are the same.

\subsection{Class field theory}

We assume that the reader is already familiar with class field theory (see \cite{ART2}, \cite{ROS}, \cite{SE1}). We recall some notation and statements. Let $k$ be a finite field. Let $K$ be a function field over $k$.

The aim of class field theory is to describe abelian extensions of $K$. Let $L/K$ be a finite abelian Galois extension of $K$ with group $G$. Class field theory associates to this extension a divisor $\mathfrak{f}(L/K)$, called the conductor. This divisor gives information about the ramified places. If $L$ is the compositum of $L_1/K$ and $L_2/K$, then one has $\mathfrak{f}(L/K)=\lcm(\mathfrak{f}(L_1/K),\mathfrak{f}(L_2/K))$. Let $M/K$ be a finite Galois extension with group $G$ and let $\chi \in \Hom(G,\C^*)$. We set $\mathfrak{f}(\chi)=\mathfrak{f}(L^{\ker(\chi)}/K)$. The set of unramified primes in $\mathcal{P}_{K/k}$ is denoted by $\unr(L/K)$. The place of degree $1$ in $\unr(L/K)$ are denoted by $\unr^1(L/K)$. For a prime $P \in\unr(L/K)$ we denote by $(P,L/K) \in G$ its Frobenius element. 
If $P \in \mathcal{P}_{L/k}$ we denote by $P|_K$ its restriction to $K$. We have a map $\Norm_{L/K}: \di_{k}^0(L) \to \di_k^0(K)$ which maps a place $P$ to $\fv(P/P|_K) P|_K$. This map induces a map $\Norm_{L/K}: \Pic^0_k(L) \to \Pic^0_k(K)$. 

Suppose $M/K$ is a finite abelian extension in some algebraic closure of $L$. Then one has for $D' \in \di_k(L)$ the equality
\begin{eqnarray*}
 (D',LM/L)|_M = (\Norm_{L/K}(D'),M/K).
\end{eqnarray*}

Class field theory gives us the following.

\begin{proposition} \label{2c843}\index{$K_{[D]}$}
 Let $K$ be a function field over $k$ and let $D \in \di_k(K)$ be of degree $1$. Then the maximal
abelian unramified extension of $K$ is the compositum of the following two disjoint extensions: $\overline{k} \cdot K$ and a unique finite
subextension
$K_{[D]}$ with Galois group isomorphic to
$\Pic_k^0(K)$ such that $(D,K_{[D]}/K)=0$. For $D' \in \di_k(K)$ we have
$(D',K_{[D]}/K)=[D']-\deg_k(D')[D] \in \Pic_k^0(K)$.
\end{proposition}

\begin{corollary} \label{2c437}
Let $k'$ be a finite extension of $k$. Let $K$ be a function field over $k$. Then the
map $\Norm_{Kk'/K}: \Pic^0_{k'}(Kk') \to \Pic^0_k(K)$ is surjective.
\end{corollary}
\begin{proof}
Proposition \ref{2c843} gives a surjective map $\Pic^0_{k'}(Kk') \to \Pic^0_k(K)$ and one easily checks that it agrees with the norm.
\end{proof}

\begin{theorem} \label{2c43}
Let $L/K$ be a geometric Galois extension of function fields over $k$ with group $G$. 
Assume that we have an injective morphism $\chi \in \Hom(G,\C^*)$. Then we have
\begin{eqnarray*}
| \sum_{P \in \unr^1(L/K)} \deg_k(P)\chi((P,L/K))| \leq m q^{1/2},
\end{eqnarray*}
where $m=2g(K)-2+\deg_k\left(\mathfrak{f}(\chi) \right)$.
It is an equality if $m=1$. 
\end{theorem}
\begin{proof}
This follows from \cite[Theorem 9.16B]{ROS}.
\end{proof}

Later we will need to compute some conductors. The following lemma is useful.

\begin{lemma} \label{2c477}
Let $K$ be a function field over $k$. Let $K_s$ be a separable closure of $K$. Let $L, M$ be finite
abelian Galois extensions of $K$ inside $K_s$ of prime degree $p$ respectively prime degree $l$ with $L \cap M=K$. Let $v \in \mathcal{P}_{K/k}$ and
suppose that $r=\mathfrak{f}(L/K)_v \in \Z_{\geq 1}$ and $s=\mathfrak{f}(M/K)_v \in \Z_{\geq 1}$. Let $w$ be the unique extension of $v$ to
$L$. Assume that $LM/L$ is ramified at
$w$ if $p=l$ and $r=s$. Then the following hold:
\begin{enumerate}
 \item $LM/K$ is totally ramified at $v$;
 \item if $p \neq l$ or $r \neq s$, we have $\mathfrak{f}(LM/L)_{w}=(p-1)\max(0,s-r)+s$;
 \item if $p=l$ and $r=s$, we have $r \geq \mathfrak{f}(LM/L)_{w} \geq t$ where $t=2$ if $p$ is the residue field characteristic of
$v$ and $1$ otherwise.
\end{enumerate}
\end{lemma}
\begin{proof}
The most important ingredient in the proof is the F\"uhrerdiskriminantenpro\-duktformel (see \cite{SE1}).
\end{proof}

One has the following lemma.

\begin{lemma} \label{2c555}
 Let $K/k$ be a function field where $k$ is a finite field. Let $L/K$ be a finite abelian Galois extension with group $G$. Let $\chi,\chi' \in
\Hom(G,\C^*)$.  
Then we have $\mathfrak{f}( \chi \cdot \chi') \leq \lcm(\mathfrak{f}(\chi),\mathfrak{f}(\chi'))$, with
equality at $P \in \mathcal{P}_{K/k}$ if we have $\mathfrak{f}(\chi)_P \neq \mathfrak{f}(\chi')_P$ or if the orders of $\chi$ and
$\chi'$ are coprime.
\end{lemma}

\subsection{Hyperelliptic curves} \label{t1}

The results of this subsection can be partially found in \cite[Subsection 7.4.3]{LIU}.
For a polynomial $f \in k[x]$ we define $f_j$ by $f=\sum_{i} f_i x^i$.

Let $k$ be a perfect field. A function field $K/k$ is called \emph{hyperelliptic} if it has full constant field $k$, if the
genus satisfies $g(K) \geq 1$, and there exists $x \in K$ with $[K:k(x)]=2$. 

Let $g \in \Z_{\geq 1}$. Consider $(f,h) \in k[x]^2$ with the following properties:
\begin{enumerate}
\item $\deg(f) \in \{2g+1,2g+2\}$ 
\item $y^2+hy-f$ is separable and irreducible in $k(x)[y]$;
\item if $\chart(k) \neq 2$ the following hold:
\begin{enumerate}
 \item $h=0$;
 \item $f$ is separable in $k[x]$;
\end{enumerate}
\item if $\chart(k)=2$, then the following hold:
\begin{enumerate}
\item $\deg(h) \leq g+1$;
\item $(h,h'^2f+f'^2)=k[x]$;
\item $(h_{g+1},h_g^2 f_{2g+2}+f_{2g+1}^2)=k$. 
\end{enumerate}
\end{enumerate}
Set $K_{f,h}=k(x)[y]/(y^2+hy-f)$ with natural inclusion $k(x) \subseteq K_{f,h}$. Then $K_{f,h}$ is a hyperelliptic curve of genus $g$. 
Furthermore, set $U'=\Spec(k[x,y]/(y^2+h(x)y-f(x))$, $V'=\Spec(k[x',y']/(y'^2+h_{\infty}(x')y'-f_{\infty}(x'))$ where
$h_{\infty}(x')=h(1/x')x'^{g+1}$ and $f_{\infty}(x')=f(1/x')x'^{2g+2}$. Let $X = U' \cup V'$ glued together by $D(x) \cong D(x')$ with relations $x=1/x'$ and $y=x^{g+1}y'$. Then $X$ is a smooth model for the curve corresponding to $K_{f,h}$. For the discriminant one has
\begin{eqnarray*}
\disc(K_{f,h}/k(x))=\left\{ \begin{array}{cc} 
\infty+(f) &  \textrm{if } \chart(k) \neq 2, \deg(f)=2g+1 \\
(f) &  \textrm{if } \chart(k) \neq 2, \deg(f)=2g+2 \\
(2g+2)\infty+2(h) & \textrm{if } \chart(k)=2.
\end{array} \right.
\end{eqnarray*}

Conversely, any hyperelliptic curve of genus $g$ has such a model.

\section{Shape parameter}

In this section, let $G$ be a finite abelian group which we denote multiplicatively. Let $\C[G]$ be the group ring of $G$ over $\C$. For $\chi \in G^{\vee}=\Hom(G,\C^*)$ and $f=\sum_{g \in G} c_g g \in \C[G]$ where $c_g \in \C$ we set
\begin{eqnarray*}
 f_{\chi}= \sum_{g \in G} c_g \chi(g^{-1}).
\end{eqnarray*}

\begin{proposition} \label{45}
Let $f=\sum_{g \in G} c_g \in \C[G]$. Then one has
\begin{eqnarray*}
f = \frac{1}{\#G} \sum_{g \in G} \sum_{\chi \in G^{\vee}} c_{\chi} \chi(g)g.
\end{eqnarray*}
\end{proposition}
\begin{proof}
This is a well-known fact and can be seen as a Fourier transform.
\end{proof}

For a subset $S \subseteq G$ we set \index{$\C[S]$} $\C[S]=\{\sum_{s \in S} c_s s: c_s \in \C\} \subseteq \C[G]$, which is a $\C$-vector space. Let
$\chi_0$ be the identity element of $G^{\vee}$. We define the shape parameter of $S$, which we denote by $\sh_G(S)$, as follows:
\begin{eqnarray*}
 \sh_G(S)=\frac{\#S}{\#G} \cdot \inf_{f \in \C[S]: f_{\chi_0} \neq 0} \frac{ \sum_{\chi \in G^{\vee}} |f_{\chi}|}{|f_{\chi_0}|}.
\end{eqnarray*}
The following proposition gives some basic properties.

\begin{proposition} \label{3cma}
Let $S \subseteq G$ be non-empty. Then the following hold:
 \begin{enumerate}
  \item For $\alpha \in \Aut(G)$ and $b \in G$ we have $\sh(b \cdot \alpha(S))=\sh(S)$. 
 
  \item We have $1 \leq \sh(S) \leq \#S$. Furthermore we have $\sh(S)=1$ if and only if $S$ is a coset of a subgroup of $G$. We have $\sh(S)=\#S$ if and only if $\#S=1$. 
  
  \item For $S \subseteq S'$ we have $\sh(S') \leq \frac{\#S'}{\#S} \sh(S)$. 
\end{enumerate}
Let $G'$ be a finite abelian group and let $S' \subseteq G'$ be non-empty. Then the following hold:
\begin{enumerate} \setcounter{enumi}{3}
  \item Let $i: G \to G'$ be an injective group morphism. Then one has $\sh_G(S)=\sh_{G'}(i(S))$.
  
  \item Let $\pi: G \to G'$ be a surjective morphism of groups. Then the equality $\sh_G(\pi^{-1}(S'))=\sh_{G'}(S')$ holds. 
  \item We have $\sh_{G \times G'}(S \times S') \leq \sh_G(S) \times \sh_{G'}(S')$. 

 \end{enumerate}
\end{proposition}
\begin{proof}
Most parts in this proof are elementary and left to the reader (see \cite{KO6}).
\end{proof}

If $S \subseteq G$ is non-empty, we set
\begin{eqnarray*}
SS^{-1}=\{st^{-1}:\ s, t \in S\}.
\end{eqnarray*}

\begin{lemma} \label{3c6}
 We have \[ \sh(S S^{-1}) \leq \frac{\#(S S^{-1})}{\#S}. \]
\end{lemma}
\begin{proof}
The function $\left(\sum_{s \in S} s \right) \cdot \left(\sum_{s \in S} s^{-1} \right)$ with support in $SS^{-1}$ gives the upper bound.
\end{proof}

\begin{definition} \label{3c540}
An \index{interval}\emph{interval} of $\Z$ is a non-empty set $S \subseteq \Z$ such that are $n,m \in \R$ with $S=[n,m] \cap \Z$. 

Let $G=\Z/n\Z$. A \emph{standard interval}\index{standard interval}\index{interval!standard} of $G$ is defined to be the image of an interval
of $\Z$ under the natural map $\Z \to \Z/n\Z$. 

Let $G$ be a finite abelian group. A subset $S \subseteq G$ is called a \emph{full interval}\index{full interval}\index{interval!full} if there exist
$n \in \Z_{\geq 1}$, a surjective
morphism
$\pi: G \to \Z/n\Z$ and a standard interval $T$ of $\Z/n\Z$ such that $\pi^{-1}(T)=S$. A full interval of a subgroup of $G$ is called an
\emph{interval}\index{interval} of $G$.
\end{definition}

\begin{lemma} \label{3c99}
For an interval $S \subseteq G$ we have $\sh(S) \leq 2$.
\end{lemma}
\begin{proof}
Using Proposition \ref{3cma}iv and v, we reduce to the case where $G=\Z/n\Z$, for which we use additive notation, and where $S$ is a standard interval. First of all assume
that the size of $S$ is odd, then we may assume (after shifting) $S=\{-\overline{m},-\overline{m}+1,\ldots,0,\ldots,\overline{m}-1,\overline{m}\}$
for some $m \in \Z_{\geq 0}$ with $m \leq \frac{n-1}{2}$. Let $T=\{0,1,\ldots,\overline{m}\}$. Then $T-T=S$
and hence we find by Lemma
\ref{3c6}
\begin{eqnarray*}
\sh(S) \leq \frac{\#S}{\#T}=\frac{2m+1}{m+1}<2.
\end{eqnarray*}
The proof in the case $\#S$ is even is similar.
\end{proof}

\begin{lemma} \label{3c900}
 Let $V$ be a vector space over $\F_p$ of dimension $n$. Let $c \in \{1, \ldots,p-1\}$ and $0 \leq i < n$ or $(c,i)=(1,n)$. Then there is an
interval $S$ in $V$ with $\#S=c p^i$.
\end{lemma}
\begin{proof}
 If $(c,i)=(1,n)$, the statement is obviously true. Assume $c \neq 0$ and let $W$ be a subspace of dimension $i+1$ of $V$ and consider a
nonzero map $f \in
W^{\vee}=\Hom(W,\F_p)$. Pick an interval $S_0$ of $\F_p$ of length $c$ and set $S=f^{-1}(S_0)$. 
\end{proof}

\section{Applications of the shape parameter to hyperelliptic curves}

\subsection{Main statements}

\begin{conv}
In this article we assume that a hyperelliptic curve $C$ of genus $g$ is given by an equation $y^2+h(x)y=f(x)$ as in Subsection \ref{t1}. Furthermore, we assume that $\infty$ is ramified in the extension $k(C)/k(x)$. This is equivalent to:
\begin{itemize}
\item $\chart(k) \neq 2$: $\deg(f)=2g+1$;
\item $\chart(k)=2$: $1 \leq \deg(h) \leq g$.
\end{itemize}
We let $\infty'$ be the point above $\infty$ of $k(C)$.
\end{conv}

Let $k$ be a finite field of cardinality $q$ and characteristic $p$. Let $C$ be a hyperelliptic curve over $k$ given by a pair $(f,h)$ following our conventions above. Then we have an injective map
\begin{eqnarray*}
\varphi_C: C(k) &\to& \Pic_k^0(C) \\
P &\mapsto& [P]-[\infty'].
\end{eqnarray*}
Let $C_S=\{P \in C(k): x(P)
\in S\}$\index{$C_S$}. We will give conditions on $\# S$ and $\sh_{k^+}(S)$ such that $\Pic^0_k(C)=\langle \varphi_C(C_S) \rangle$.

 Let $P \neq \infty$ be a prime of $k(x)$, the function field of the projective line over $k$, corresponding to the monic polynomial $\sum_{i=0}^n a_i x^i$ with $a_n=1$.
We put $\Tv(P)=-a_{n-1} \in k$. 

\begin{proposition} \label{4c923}
 Assume that $p=2$ and that $\deg(h)=g$. Then we have a surjective morphism of
groups
\begin{eqnarray*}
\psi_C: \Pic^0_k(C) \to \F_2
\end{eqnarray*}
defined as follows:
Let $P \neq \infty'$ be a prime of $k(C)$. Then we have:
\begin{eqnarray*}
\psi_C([P]-\deg_k(P)[\infty']) = \tr_{k/\F_2}\left( \frac{\fv(P/P|_K)\Tv(P|_K)d_1+\deg_k(P)d_0}{h_g^2} \right) \in \F_2.
\end{eqnarray*}
\end{proposition}

\begin{theorem}\label{4c911}
Let $C$ over $k$ be a hyperelliptic curve of genus $g$ given according to our assumptions as above such that $\#C(k)>(2g-2)\sqrt{q}$. Put $s=2$ if
$p=2$ and $s=3$ if $p \neq 2$. Let $S \subseteq k^+$ such that
\begin{eqnarray*}
q^{3/2} \cdot 2(2g-2+s) \cdot \sh_{k}(S)< \left( \#C(k)+(2g-2+2s)\sqrt{q} \right) \cdot \#S.
\end{eqnarray*}

Then the following hold:
\begin{enumerate}
 \item 
Assume that $p \neq 2$ or $p=2$ and $\deg(h)<g$. Then we have $\langle \varphi_C(C_S) \rangle=\Pic^0_k(C)$. 
 \item
Assume that $p=2$ and $\deg(h)=g$. Define the following:
\begin{eqnarray*}
d_i &=& f_{2g+i}+\sqrt{f_{2g+2}}h_{g-1+i} \in k \ (\textrm{for }i \in \{0,1\}) \\
\epsilon_C &=& (-1)^{\tr_{k/\F_2}(d_0/h_g^2)} \in \C \\
 \lambda_2 &\in& \Hom(k^+,\C^*),\ c \mapsto (-1)^{\tr_{k/\F_2}(cd_1/h_g^2)} \\
H_C &=& \{x \in k: \lambda_2(x)=-\epsilon_C\} \subseteq k.
\end{eqnarray*}
Then we have:
\begin{enumerate}
 \item $\langle \varphi_C(C_S) \rangle \in \{\Pic_k^0(C),\ker(\psi_C)\}$;
 \item if $S \cap H_C= \emptyset$, then $\langle \varphi_C(C_S) \rangle=\ker(\psi_C)$;
 \item if $S \cap H_C \neq \emptyset$, then $\langle \varphi_C(C_S) \rangle=\Pic_k^0(C)$ if
\begin{eqnarray*}
q^{3/2}(2g-2+s)\sh_k(S \cap H_C)< \left( \#C(k)+(2g-2+2s) \sqrt{q}  \right) \cdot \# (S \cap H_C).
\end{eqnarray*}
\end{enumerate}
\end{enumerate}
\end{theorem}

\begin{remark}
Similar results can be obtains for $S \subseteq k^*$ when one takes the shape with respect to $k^*$.
\end{remark}

\begin{remark} \label{4c945}
Theorem \ref{4c911} depends on $C$ because of the dependence on $\# C(k)$. By Hasse-Weil we have $\#C(k) \geq q+1-2g \sqrt{q}$ and we can get rid of this dependence.
\end{remark}

From the above theorem we deduce one of the theorems of the introduction.

\begin{proof}[of Theorem \ref{4c845}]
 This follows from Theorem \ref{4c911}, Hasse-Weil and bounds on the shape (Proposition \ref{3cma}ii and Lemma \ref{3c99}).
\end{proof}

\begin{example}
Assume that $g=1$ in Theorem \ref{4c911}. Then using some crude estimates, one sees that we can apply the theorem if $2s \cdot \sh_{k^+} (S)
\leq
\#S$. Furthermore, the exceptional case corresponds to ordinary elliptic curves in characteristic $2$. In this case, there is a unique subgroup of
$\Pic^0_k(E) \cong E(k)$ of index $2$, namely $2E(k)$, which must be equal to $\ker(\psi_E)$. 
\end{example}

\subsection{Realizing Galois groups}

The goal in this subsection is to realize $k^+ \times \Pic^0_k(C)$ as the Galois group of an extension $M$ of $k(C)$ such that for $P 
\in \mathcal{P}_{k(C)/k}$ of degree $1$ we have $(P,M/k(C))=(x(P),[P]-[\infty'])$.

Let us realize $k^+$ first. Set $K=k(x)$.

\begin{proposition} \label{4c38}
Let $K_+=K[Y]/(Y^q-Y-x)$ and let $y=\overline{Y} \in K_+$. Then $K_+/K$ is a Galois extension of fields for which the following hold:
\begin{enumerate}
 \item the map $\varphi: k \to \Gal(K_+/K)$, $c \mapsto (y \mapsto y+c)$ is an isomorphism of groups;
 \item the extension is totally ramified at $\infty$, and is unramified at all the other primes;
 \item the extension is geometric;
 \item for $P \in \mathcal{P}_{K/k} \setminus \{\infty\}$ we have $(P,K_+/K)=\varphi(\Tv(P)) \in \Gal(K_+/K)$;
 \item $\mathfrak{f}(K_+/K)=2 \infty$, $\disc(K_+/K)=2(q-1)\infty$; the conductor of any nontrivial subextension
of $K_+/K$ is $2\infty$;
 \item $g(K_+)=0$.
\end{enumerate}
\end{proposition}
\begin{proof}
This is a calculation which involves Riemann-Hurwitz (see \cite{ST}) and the F\"uhrerdiskriminantenproduktformel (see \cite{SE1}).
\end{proof}

\begin{proposition} \label{4c301}
 Let $K_+$ be as in the previous proposition (Proposition \ref{4c38}). For $c \in k^*$ put $z_{c}=(cy)+(cy)^p+(cy)^{p^2}+\ldots+(cy)^{p^{m-1}}$. For
$\overline{c} \in k^*/\F_p^*$ set $K_{\overline{c}}=K(z_c)$. Let $\tau_c: k \to \F_p^*$ be defined by $a \mapsto  \tr_{k/\F_p}(ca)$. Then
the following hold:
\begin{enumerate}
 \item $z_c$ is a zero of the irreducible polynomial $f_c=X^p-X-cx \in k(x)[X]$;
 \item $K_{\overline{c}}/K$ is Galois, the map $\varphi_c: \F_p \to \Gal(K_{\overline{c}}/K)$, $a \mapsto (z_c \mapsto z_c+a)$ is an isomorphism and
the
following diagram is commutative:
\[
\xymatrix{
\Gal(K_+/K) \ar[r]^{\sim} \ar@{->>}[d] & k \ar@{->>}[d]^{\tau_c} \\
\Gal(K_{\overline{c}}/K) \ar[r]^{\sim} & \F_p;
}  
\]
\item for $P \in \mathcal{P}_{K/k} \setminus \{\infty\}$ we have \[(P,K_{\overline{c}}/K)= \varphi_c(\tr_{k/\F_p}(c\Tv(P))) \in
\Gal(K_{\overline{c}}/K);\]
\item the map
\begin{eqnarray*}
 k^*/\F_p^* &\to& \{L: K \subseteq L \subseteq K_+, [L:K]=p\} \\
\overline{c} &\mapsto& K_{\overline{c}}
\end{eqnarray*}
is a bijection.
\end{enumerate}
\end{proposition}
\begin{proof}
This is a calculation which follows easily from Proposition \ref{4c38}. 
\end{proof}

Proposition \ref{2c843} gives us an extension $k(C)_{[\infty']}/k(C)$ which is unramified with Galois group $\Pic^0_k(C)$ and the Frobenius
at a rational point is $[P]-[\infty']$. Proposition \ref{4c38} gives us an extension $K_+/K$ with group $k^+$. Consider the
following diagram
of function fields:
\[
\xymatrix{
 & & k(C)_{+,[\infty']}=k(C)_+ k(C)_{[\infty']} & \\
 & k(C)_+=K_+k(C) \ar[ru] & & k(C)_{[\infty']} \ar[lu] \\
 K_+ \ar[ru] & & k(C) \ar[lu] \ar[ru]_{\Pic^0_k(C)} & \\
 & K=k(x). \ar[lu]^{k^+} \ar[ru]_{C_2} & &
}\]

We will first study $\Gal(k(C)_+/K)$. 

First of all, the extension $k(C)/K$ is Galois with group $C_2$ and totally ramified at $\infty$ and at some
more points. The extension $K_+/K$ is geometric and Galois with group $k^+$ and totally ramified at $\infty$. Consider the extension
$k(C)_+/K$. As $K_+$ and $k(C)$ are linearly disjoint by genus considerations (Riemann-Hurwitz), $k(C)_+/K$ is Galois with group $k^+ \times C_2$. Also $k(C)_+/k(C)$ is Galois
with group $k^+$. We claim that $k(C)_+/K$ is geometric. 
If $\chart(k) \neq 2$, then as $(\# k,2)=1$, the extension $k(C)_+/K$ is totally ramified at $\infty$. Assume that $\chart(k)=2$ and that $\deg(h)
<g$. The conductor at $\infty$ of $k(C)/K$ is $2(g+1-\deg(h)) \infty$, 
which is more than the conductor of $K_+/K$ at $\infty$, which is $2 \infty$. Hence $k(C)_+/K$ is totally ramified at $\infty$ and $k(C)_+/k(C)$ is
totally ramified at $\infty'$.  
Assume that $\chart(k)=2$ and that $\deg(h)=g$. In this case, take a prime of $K$, not $\infty$, dividing $h$. Then $k(C)/K$ is ramified at this
prime, but $K_+/K$ is not. Hence $k(C)_+/K_+$ is ramified at a prime above such a
prime, and it cannot be a constant field extension. We conclude that $k(C)_+/K$ is always geometric. 

The only possible ramification in $k(C)_+/k(C)$ is at $\infty'$. We have already shown that it is totally
ramified at $\infty'$ if $\chart(k) \neq 2$ or $\chart(k)=2$ and $\deg(h)=g$.
One knows that the maximal abelian extension of $K_{\infty}$, the completion of $K$ at $\infty$, which is totally ramified of conductor $2$ has
degree $q$. Hence if $\chart(k)=2$ and $\deg(h)=g$, we see that $k(C)_+/K$ cannot be totally ramified at $\infty$. Hence in this case $k(C)_+/k(C)$
cannot be totally ramified. There is a unique field $L$ with $k(C) \subseteq
L \subseteq k(C)_+$
with $[L:k(C)]=2$ which is unramified at $\infty'$, and hence unramified. 

\begin{lemma} \label{4c999}
 Let $k$ be a finite field of characteristic $p$ and let $a \in k$. Then $f_a=x^p-x-a \in k[x]$ is irreducible if and only if
$\tr_{k/\F_p}(a) \neq 0$.
\end{lemma}
\begin{proof}
We leave the proof as an exercise for the reader (see \cite{KO6}).
\end{proof}

The following lemma explicitly describes $L$.

\begin{lemma} \label{4c438}
 Assume that $p=2$ and that $\deg(h)=g$. For $i=0,1$ put $d_i=f_{2g+i}+\sqrt{f_{2g+2}}h_{g-1+i}$.
Then the unique unramified subextension $L$ of $k(C)_+/k(C)$
comes from the subextension of $K_+/K$ given by $z^2-z-cx$ with $c=\frac{d_1}{h_g^2}$. This extension
$L/k(C)$ is totally split at $\infty'$ if and only if $\tr_{k/\F_2}(\frac{d_0}{h_g^2}) =0$.
\end{lemma}
\begin{proof}
Let $v$ be the normalized valuation at $\infty'$ of $k(C)$. Then $v(x)=-2$ as $k(C)/K$ is ramified. We have $\deg(f) \in \{2g+1,2g+2\}$. 
Put $y'=y+\sqrt{f_{2g+2}}x^{g+1} \in k(C)$. Then we have $y'^2+hy'=f_{\new}$ where $f_{\new}=f+f_{2g+2}x^{2g+2}+\sqrt{f_{2g+2}}
hx^{g+1}$. Note that $f_{\new,2g+1}=d_1$ is nonzero, as its square is nonzero by our assumptions on $(f,h)$. Hence $f_{\new}$ is of degree $2g+1$. From the equation which $y'$ satisfies, one easily obtains $v(y')=-(2g+1)$. 

 Let $z$ be an element of $K_+$ satisfying $z^2-z-d_1x/h_g^2=0$ (Proposition \ref{4c38} for the existence). 
Notice that $y''=y'/(h_gx^g)$ satisfies
\begin{eqnarray*}
 y''^2+y'' &=& \frac{f_{\new}(x)+(h(x)-h_gx^g)y'}{h_g^2 x^{2g}} \\
&=& \frac{d_1x}{h_g^2}+\frac{d_0}{h_g^2}+ \frac{(f_{\new}-d_1x^{2g+1}-d_0x^{2g})+(h(x)-h_gx^g)y'}{h_g^2x^{2g}}.
\end{eqnarray*}
Hence we have
\begin{eqnarray*}
 (y'+z)^2+(y'+z) = \frac{d_0}{h_g^2}+\frac{(f_{\new}(x)-f_{2g+1}x^{2g+1}-f_{2g}x^{2g})+(h(x)-h_gx^g)y'}{h_g^2x^{2g}}.
\end{eqnarray*}
The valuation of the right hand side at infinity is non-negative and the part in the fraction has a positive valuation. The theorem of Kummer (\cite[Chapter 2, Theorem 3.7]{ST}) gives the follwing. It shows that the extension $L/k(C)$ is unramified at infinity, and that the extension splits completely at infinity if and
only if the
polynomial
$x^2+x+\frac{d_0}{h_g^2}$ is not irreducible in $k[x]$. This happens if and only if $\tr_{k/\F_2}(\frac{d_0}{h_g^2}) =0$ by Lemma
\ref{4c999}. 
\end{proof}

The following lemma gives us the conductor of subextensions of $k(C)_+/k(C)$.

\begin{lemma} \label{4c754}
Let $L'$ be a subextension of degree $p$ of $k(C)_+/k(C)$ which is totally ramified at $\infty'$. Then one has 
\begin{eqnarray*}
\mathfrak{f}(L'/k(C))= \left\{ \begin{array}{cc} 2\infty' & p=2 \\
          3 \infty' & p \neq 2.
         \end{array} \right.
\end{eqnarray*}
\end{lemma}
\begin{proof}
This follows from Lemma \ref{2c477} and Proposition \ref{4c38}.
\end{proof}

The next step is to study the extension $k(C)_{+,[\infty']}/k(C)$. If $p \neq 2$ or $p=2$ and $\deg(h)<g$, then we have seen above that $k(C)_+/k(C)$
is totally ramified at
$\infty'$. As $k(C)_{[\infty']}/k(C)$ is unramified, it shows that $k(C)_+$ and $k(C)_{[\infty']}$ are disjoint over $k(C)$. In this case we have
$\Gal(k(C)_{+,[\infty']}/k(C))=k^+
\times \Pic^0_k(C)$. 

 Assume that $p=2$ and that $\deg(h)=g$. We want to understand the Galois extension $k(C)_{+,[\infty']}/k(C)$. Using Lemma \ref{4c438} and Proposition
\ref{2c843}, we
see that two things can happen:
If $\tr_{k/\F_2}(\frac{d_0}{h_g^2}) =0$, then one obtains $L \subseteq k(C)_{[\infty']}$ (there is a unique maximal extension where $\infty'$ splits). This means that there is a surjective homomorphism $\Pic^0_k(C) \to
\Gal(L/k(C))$. One has $\Gal(k(C)_{+,[\infty']}/k(C))=k^+ \times_{\Gal(L/k(C))} \Pic_k^0(C)$.
If $\tr_{k/\F_2}(\frac{d_0}{h_g^2}) =1$, then $k(C)_+$ and $k(C)_{[\infty']}$ are disjoint, and $\Gal(k(C)_{+,[\infty']}/k(C))=k^+ \times
\Pic^0_k(C)$.
Unfortunately, the extension is not geometric. There is a degree $2$ extension of $k$ inside $k(C)_{+,[\infty']}$ (Proposition \ref{2c843}). Also in this case one can produce a surjective homomorphism $\Pic^0_k(C) \to \Gal(L/k(C))$. 

\begin{proof}[of Proposition \ref{4c923}]
 Assume first $\tr_{k/\F_2}(\frac{d_0}{h_g^2})=0$. Then $L \subseteq k(C)_{[\infty']}$ and this gives a surjective map $\psi_C$ on the Galois
groups. To see what it does, we look at the Frobenius elements. Let $P$ be a prime of degree $n$ in $k(C)$. One has
$(P,k(C)_{[\infty']}/k(C))=[P]-n[\infty'] \in \Pic^0_k(C)$ (Proposition \ref{2c843}). This Frobenius maps to
$(P,L/k(C))=\tr_{k/\F_2}(\frac{\fv(P/P|_K)\Tv(P|_K) d_1}{h_g^2})=\tr_{k/\F_2}(\frac{\fv(P/P|_K)\Tv(P|_K)d_1+\deg_k(P)d_0}{h_g^2})$ (Proposition \ref{4c301} and Lemma \ref{4c438}).

Assume $\tr_{k/\F_2}(\frac{d_0}{h_g^2})=1$. Let $L'$ be the third degree $2$ extension in the $V_4$
extension $L k'$ over $k(C)$ where $k'$ is the unique degree $2$ extension of $k$. Then we have a natural map $\Pic^0_k(C) \to
\Gal(L'/k(C))=\F_2$ (Proposition \ref{2c843}). Let $P$ be a prime of $k(C)$ of degree $n$. Note that there is a unique maximal extension
in $Lk'/k(C)$ where $P$ is totally split. Assume that $n$ is even. Then $P$ splits in $L'/k(C)$ iff it splits
in $L/k(C)$. If $n$ is odd, then $P$ splits in $L'/k(C)$ iff it does not
split in $L/k(C)$. This gives the required map.
\end{proof}

\subsection{Character sum estimates} \label{4c555}
Put 
\begin{eqnarray*}
C(k)^*=C(k) \setminus \{\infty'\}=\unr^1(k(C)_{+,[\infty']}/k(C)).
\end{eqnarray*}
Let $\lambda \in k^{\vee}$ and $\chi \in \Pic^0_k(C)^{\vee}$.
Since we have a natural map
$\Gal(k(C)_{+,[\infty']}/k(C)) \to
k^+ \times \Pic^0_k(C)$, we can view $(\lambda,\chi)$ as a character of $\Gal(k(C)_{+,[\infty']}/k(C))$ by taking the product. We put 
\begin{eqnarray*}
c_{(\lambda,\chi)}= \sum_{P \in C(k)^*}(\lambda,\chi)(P,k(C)_{+,[\infty']}/k(C))= \sum_{P \in C(k)^*}\lambda(x(P)) \chi(\varphi_C(P))
\end{eqnarray*}
(we avoid the only ramification at $\infty'$). Our goal is to estimate these $c_{(\lambda,\chi)}$. 
Put $s=2$ if $p=2$ and $s=3$ if $p \neq 2$. 

\subsubsection{Case 1} Assume that $p \neq 2$ or $p=2$ and $\deg(h)<g$.

\begin{lemma} \label{4c111}
 The following hold for $\lambda \in k^{\vee}$ and $\chi \in \Pic^0_k(C)^{\vee}$.
\begin{enumerate}
 \item if $\lambda \neq \chi_0$, then $|c_{(\lambda,\chi)}| \leq (2g-2+s) \sqrt{q}$;
 \item $c_{\chi_0,\chi_0} = \#C(k)-1$;
 \item if $\chi \neq \chi_0$, then $|c_{(\chi_0,\chi)}+1| \leq (2g-2) \sqrt{q}$.
\end{enumerate}
\end{lemma}
\begin{proof}
 i: The degree of the conductor of the corresponding extension is $s$ (see Lemma \ref{4c754} and Lemma \ref{2c555}). Hence the result
follows from Theorem \ref{2c43}.

 ii: Obvious.

iii: The degree of the conductor of the corresponding extension is $0$ (Lemma \ref{4c754}, Lemma \ref{2c555}). Hence the result follows
from Theorem \ref{2c43}.
\end{proof}

\subsubsection{Case 2}
Assume that $p=2$ and $\deg(h)=g$.
 Let $\lambda_2$ be the special character of $k^+$ corresponding to the unramified
subextension of $L/k(C)$ of degree $2$. More explicitly, we define
$\lambda_2 \in  k^{\vee}$, $c \mapsto (-1)^{\tr_{k/\F_2}(cd_1/h_g^2)} \in \C^*$ (Lemma \ref{4c438} and Proposition \ref{4c301}).
Put $\epsilon_C = (-1)^{\tr_{k/\F_2}(d_0/h_g^2)}$ (it is $-1$ if there is a constant
field extension). Let $\chi_2=(-1)^{\psi_C} \in \Pic^0_k(C)^{\vee}$.

\begin{lemma} \label{4c1112}
 The following hold for $\lambda \in k^{\vee}$ and $\chi \in \Pic^0_k(C)^{\vee}$.
\begin{enumerate}
 \item $c_{(\lambda,\chi)\cdot (\lambda_2,\chi_2)}=\epsilon_C c_{(\lambda,\chi)}$;
 \item if $\lambda \neq \chi_0, \lambda_2$, then $|c_{(\lambda,\chi)}| \leq (2g-2+s) \sqrt{q}$;
 \item $c_{(\chi_0,\chi_0)} = \# C(k)-1$;
 \item $c_{(\lambda_2,\chi_2)} = \epsilon_C \left( \# C(k)-1 \right)$;
 \item if $\chi \neq \chi_0$, then $|c_{(\chi_0,\chi)}+1| \leq (2g-2)\sqrt{q}$;
 \item if $\chi \neq \chi_2$, then $|c_{(\lambda_2,\chi)}+\epsilon_C| \leq (2g-2) \sqrt{q}$.
\end{enumerate}
\end{lemma}
\begin{proof}
i: Let $P \in C(k)^*$. We have $\lambda_2(x(P))\chi_2(\varphi_C(P))=\epsilon_C$ by construction. Indeed, if $\epsilon_C=1$, $\lambda_2(x(P))$ and
$\chi_2(\varphi_C(P))$ are equal. If $\epsilon_C=-1$, then a rational point splits in one extension iff it does not split in the other one, and hence
they differ by a sign. The result follows.

 ii: The degree of the conductor of the corresponding extension is $s$ (Lemma \ref{4c754}, Lemma \ref{2c555}). Hence the result
follows from Theorem \ref{2c43}.

iii: Obvious.

iv: Follows from ii and i.

v:  The degree of the conductor of the corresponding extension is $0$ (Lemma \ref{4c754}, Lemma \ref{2c555}). Hence the result follows
from Theorem \ref{2c43}.

vi: Follows from v and i.
\end{proof}

\subsection{Proof of theorem}

\begin{proof}[of Theorem \ref{4c911}]
Suppose $\langle \varphi_C(C_S) \rangle \subsetneq \Pic^0_k(C)$. Then
there exists a subgroup $H \subseteq \Pic^0_k(C)$ of prime index $l$ such that $\varphi_C(C_S) \subseteq H$. Let $\chi \in \Pic^0_k(C)^{\vee}$
be a character with kernel $H$. Let $f= \sum_{a \in k} f_a a \in \C[S] \subseteq \C[k]$. Then for $a \in k$ we have $f_a=\frac{1}{q} \sum_{\lambda \in k^{\vee}} f_{\lambda} \lambda(a)$ (Proposition \ref{45}).

By construction we have 
\begin{eqnarray*}
0&=&\sum_{P \in C(k)^*} f(x(P)) (\chi-1)(\varphi_C(P))= \frac{1}{q} \sum_{P \in C(k)^*} \sum_{\lambda \in
k^{\vee}}
f_{\lambda} \lambda(x(P)) (\chi-1)(P)\\
&=& \frac{1}{q}\sum_{\lambda \in k^{\vee}} f_{\lambda} \left( c_{(\lambda,\chi)}-c_{(\lambda,1)} \right). 
\end{eqnarray*}

Assume that $\chi \neq \chi_2$ if $p=2$ and $\deg(h)=g$. Choose $f$ such that
$\sh_k(S)=\#S/q \cdot C_k(f)$. Rewrite our
equation in the following way:
\begin{eqnarray*}
 f_1(c_{(1,1)}-c_{(1,\chi)}) = \sum_{\lambda \in k^{\vee}, \lambda \neq 1} f_{\lambda} \left( c_{(\lambda,\chi)}-c_{(\lambda,1)} \right).
\end{eqnarray*}
We will now put in the estimates of Lemma \ref{4c1112}. Notice first
\begin{eqnarray*}
 |c_{(1,1)}-c_{(1,\chi)}|&=&|(c_{(1,1)}+1)-(c_{(1,\chi)}+1)|=|\#C(k)-(c_{(1,\chi)}+1)| \\
&\geq& \#C(k)-(2g-2)\sqrt{q}>0.
\end{eqnarray*}
Taking absolute values gives
\begin{eqnarray*}
 |f_1|(\#C(k)+(2g-2+2s)\sqrt{q}) \leq 2(2g-2+s) \sqrt{q} \sum_{\lambda \in k^{\vee}} |f_{\lambda}|.
\end{eqnarray*}
Pick $f$ such that $C(f)=q/\#S \cdot \sh_k(S)$. Then we obtain
\begin{eqnarray*}
\frac{q}{\#S} \cdot \sh_k(S) = C(f) \geq \frac{\#C(k)+(2g-2+2s)\sqrt{q}}{2(2g-2+s) \sqrt{q}}
\end{eqnarray*}
and this gives us the required result. 

Assume that $p=2$, $\deg(h)=g$ and that $\chi=\chi_2$. Then one has
\begin{eqnarray*}
 0= \sum_{\lambda \in k^{\vee}} f_{\lambda} \left( c_{(\lambda,\chi_2)}-c_{(\lambda,1)} \right) = \sum_{\lambda \pmod{\langle \lambda_2 \rangle}}
\left( f_{\lambda}-\epsilon_C f_{\lambda \lambda_2} \right) \left( c_{(\lambda,\chi_2)}-c_{(\lambda,1)} \right). 
\end{eqnarray*}
Hence we have
\begin{eqnarray*}
(f_1-\epsilon_C f_{\lambda_2}) \left( c_{(1,1)}-c_{(1,\chi_2)} \right) = 1/2 \sum_{\lambda \in k^{\vee}, \lambda \neq 1,\lambda_2}
\left( f_{\lambda}-\epsilon_C f_{\lambda \lambda_2} \right) \left( c_{(\lambda,\chi_2)}-c_{(\lambda,1)} \right).
\end{eqnarray*}
The estimates of Lemma \ref{4c1112} give
\begin{eqnarray*}
 |f_1-\epsilon_C f_{\lambda_2}| (\#C(k)+(2g-2+2s) \sqrt{q}) \leq (2g-2+s) \sqrt{q} \sum_{\lambda \in k^{\vee}} |f_{\lambda}-\epsilon_C f_{\lambda
\lambda_2} |.
\end{eqnarray*}
Consider the expression $f_{\lambda} - \epsilon_C f_{\lambda \lambda_2}$. It is not hard to see that the image of the map $\C[S] \to \C[S]$, $f \mapsto f_{\lambda} - \epsilon_C f_{\lambda \lambda_2}$ is $\C[H_C \cap S]$ where $H_C= \{x \in k: \lambda_2(x)=-\epsilon_C\}$.
If $H_C \cap S = \emptyset$, then we have $C_S \subseteq \ker(\psi_C)
\stackrel{2}{\subseteq} \Pic^0_k(C)$ (Proposition \ref{4c923}). 
We can interpret our equation as a shape of $H_C \cap S$ and by choosing the function which obtains the shape of $H_C \cap S$ we obtain:
\begin{eqnarray*}
\frac{\#C(k)+(2g-2+2s) \sqrt{q}}{(2g-2+s)\sqrt{q}} \leq \frac{q}{\# (S \cap H_C)} \cdot \sh_k(S \cap H_C).
\end{eqnarray*}
\end{proof}

\section{The algorithm}

In this section we will describe how to find generators for $\Pic^0_k(C)$, that is, we give the proof of Theorem \ref{4cext}. We make a few
assumptions:
\begin{enumerate}
 \item We can do operations in $k$, a finite field of cardinality $q$, as addition and multiplication in time polynomial in $\log(q)$.
 \item Our hyperelliptic curve $C$ is given as in Subsection \ref{t1} and $k(C)/k(x)$ is totally ramified at
$\infty$.
 \item Divisors on $\Pic^0_k(C)$ are represented as
Galois-invariant divisors of $\di_{\overline{k}}^0(C_{\overline{k}})$, where divisors on
$\di^0_{\overline{k}}(C_{\overline{k}})$ are represented in $\Z^{(C(\overline{k}))}$.
\end{enumerate}

\begin{proof}[of Theorem \ref{4cext}]
Put $t=\left(2^4(2g+1)+2^2 \right)^2$. Deterministically construct $k'$, a finite field extension of $k$ of cardinality $q^i$ where $tq > q^i
\geq t$. This can be done in time $O(q^{1/2}i^4)$ (\cite{SHO2}), which is in $O(q^{1/2}g^2)$. Addition and multiplication can then be done in $k'$ in
time polynomial in $\log(g)$ and $\log(q)$.

Construct an interval $S$ of $k'$ with the following properties:
\begin{enumerate}
 \item $\#S \geq \lceil 4(2g+1) q^{i/2} \rceil = r$;
 \item $\#S=O(g^2 q^{1/2})$;
 \item if $p=2$ and $\deg(h)=g$, then $S \subseteq H_C$ (see Theorem \ref{4c911}).
\end{enumerate}
This can be done for the following reason. We claim that there are intervals of length between $r$ and $2r$. Indeed, write $r$ in basis
$p=\chart(k)$, say with main term $a_s p^s$. We claim that there is an interval of cardinality $r'=2a_s p^s$. Note that $r \leq r' \leq 2r$. We want
to apply Lemma \ref{3c900} (for $H_C$ in the special case), and for this it is enough to show that $4r\leq q^i$. 
Indeed, we have
\begin{eqnarray*}
q^i \geq q^{i/2} t^{1/2}=q^{i/2}\left(2^4(2g+1)+2^2 \right) \geq 4( 4(2g+1)q^{i/2}+1 ) \geq 4r.
\end{eqnarray*}
We claim that $\#S=O(g^2 q^{1/2})$. Indeed, $g q^{i/2} \leq g t^{1/2} q^{1/2}$, which is of order $O(g^2 q^{1/2})$ and the result follows. 

We will apply Theorem \ref{4c911} with our interval $S$. We have $\sh_{k^+}(S) \leq 2$ (Lemma \ref{3c99}) and $q^i \geq (4g-2)^2$ and
hence
$\#S \geq 2
\sh_{k+}(S)(2g-2+s)q^{i/2}$. Theorem \ref{4c911} (see Remark \ref{4c945}) gives $\langle \varphi_{C_{k'}}(C_{k',S}) \rangle=
\Pic^0_{k'}(C_{k'})$. 

We will construct $C_{k',S}$. For all $x \in S$ we look at the equation $y^2+h(x)y=f(x)$ and we have to solve this in $y$ (note that we have a smooth model of $C$). 

Assume that $p=2$. Note that $h \neq 0$. If $x$ is fixed, we need
to find $y$ with 
\begin{eqnarray*}
 \left(\frac{y}{h(x)} \right)^2-\frac{y}{h(x)}=\left(\frac{f(x)}{h(x)} \right)^2.
\end{eqnarray*}
This is an Artin-Schreier equation and solutions can easily be obtained by linear algebra. Each step here can be done in polynomial time in
$O(q^i)$, hence polynomial time in $\log(g)$ and $\log(q)$. Hence the total cost of this is
$O(g^{2+\epsilon}q^{1/2+\epsilon})$.

Assume that $p \neq 2$. Then for $x \in S$ we need to solve $y^2=f(x)$. First calculate a quadratic
non-residue in time $O(q^{i/4+\delta})$, that is, in time $O(\log(g)^{1/2}q^{1/4+\delta})$ (see \cite{SHP2}). Then we apply Tonelli-Shanks to solve
the equation for a fixed $x$ in time polynomial in $\log(q)$ (\cite[Lemma 3.4]{WOE}). Hence in total the cost of this step is
again $O(g^{2+\epsilon}q^{1/2+\epsilon})$.

Hence we have calculated $C_{k',S}$. Let $\infty''$ be the point at infinity of $C_{k'}$. The image of $C_{k',S}$ under $\varphi_{C_{k'}}: C_{k'}(k') \to
\Pic^0_k(C_{k'})$ generates the group $\Pic^0_k(C_{k'})$. It maps $P$ to $[P]-[\infty'']$. Since the norm map $\Norm_{k'k(C)/k(C)} : \Pic^0_k(C_{k'})
\to \Pic^0_k(C)$ is surjective (Corollary \ref{2c437}), a generating set of $\Pic^0_k(C)$ is given by 
\begin{eqnarray*}
\Norm_{k'k(C)/k(C)} \left(\varphi_{C_{k'}}(C_{k',S}) \right).
\end{eqnarray*}
More explicitly, for $P \in C_{k',S}$ we have
\begin{eqnarray*}
\Norm_{k'k(C)/k(C)}(\varphi_{C_{k'}}(P)) = - [k':k][\infty']+\sum_{g \in \Gal(k'/k)} [g(P)]. 
\end{eqnarray*}
\end{proof}

\bibliographystyle{acm}


\end{document}